\documentclass[10pt]{article}

\usepackage{graphics}
\usepackage{epsfig}
\usepackage{amssymb}
\usepackage{amsthm}
\usepackage{mathrsfs}
\usepackage{hhline}
\usepackage{amsmath}
\usepackage{array}
\usepackage{multirow}
\usepackage{float}
\usepackage{picinpar}
\usepackage{cite}
\usepackage{enumerate}
\usepackage{xcolor}
\newcommand{\round}[1]{\mbox{\textcircled{\scriptsize{$#1$}}}}
\makeatletter
\newcommand{\thickhline}{%
    \noalign {\ifnum 0=`}\fi \hrule height 1pt
    \futurelet \reserved@a \@xhline
}
\newcolumntype{"}{@{\hskip\tabcolsep\vrule width 1pt\hskip\tabcolsep}}
\makeatother

\usepackage[mathlines]{lineno}
\modulolinenumbers[2]
\newcommand*\patchAmsMathEnvironmentForLineno[1]{%
\expandafter\let\csname old#1\expandafter\endcsname\csname #1\endcsname  \expandafter\let\csname oldend#1\expandafter\endcsname\csname end#1\endcsname  \renewenvironment{#1}%
{\linenomath\csname old#1\endcsname}%
{\csname oldend#1\endcsname\endlinenomath}}%
\newcommand*\patchBothAmsMathEnvironmentsForLineno[1]{%
\patchAmsMathEnvironmentForLineno{#1}%
\patchAmsMathEnvironmentForLineno{#1*}}%
\AtBeginDocument{%
\patchBothAmsMathEnvironmentsForLineno{equation}%
\patchBothAmsMathEnvironmentsForLineno{align}%
\patchBothAmsMathEnvironmentsForLineno{flalign}%
\patchBothAmsMathEnvironmentsForLineno{alignat}%
\patchBothAmsMathEnvironmentsForLineno{gather}%
\patchBothAmsMathEnvironmentsForLineno{multline}%
}

\def\red{\color{red}}

\newtheorem{theorem}{Theorem}[section]
\newtheorem{lemma}[theorem]{Lemma}
\newtheorem{example}{Example}[section]
\newtheorem{corollary}[theorem]{Corollary}
\newtheorem{problem}{Problem}[section]

\newtheorem{conjecture}{Conjecture}[section]

\newtheorem{rem}{Remark}[section]
\numberwithin{equation}{section}
\def\Z{\mathbb{Z}}

\begin{document}
\normalsize
\begin{center}
{\Large \bf Cartesian Magicness of 3-Dimensional Boards}

\bigskip
{\large Gee-Choon Lau\footnote{E-mail: geeclau@yahoo.com}}\\
\emph{Faculty of Computer \& Mathematical Sciences,}\\
\emph{Universiti Teknologi MARA (Johor Branch),}\\
\emph{85000, Segamat, Malaysia.}\\
\medskip

{\large Ho-Kuen Ng\footnote{E-mail: ho-kuen.ng@sjsu.edu}}\\
\emph{Department of Mathematics, San Jos\'{e} State University,}\\
\emph{San Jose CA 95192 USA.}\\
\medskip

{\large Wai-Chee Shiu\footnote{E-mail: wcshiu@hkbu.edu.hk}}\\
\emph{Department of Mathematics, Hong Kong Baptist University,}\\
\emph{224 Waterloo Road, Kowloon Tong, Hong Kong, P.R. China.}
\end{center}

\begin{abstract}
A $(p,q,r)$-board that has $pq+pr+qr$ squares consists of a $(p,q)$-, a $(p,r)$-, and a $(q,r)$-rectangle. Let $S$ be the set of the squares. Consider a bijection $f : S \to [1,pq+pr+qr]$. Firstly, for $1 \le i \le p$, let $x_i$ be the sum of all the $q+r$ integers in the $i$-th row of the $(p,q+r)$-rectangle. Secondly, for $1 \le j \le q$, let $y_j$ be the sum of all the $p+r$ integers in the $j$-th row of the $(q,p+r)$-rectangle. Finally, for $1\le k\le r$, let $z_k$ be the the sum of all the $p+q$ integers in the $k$-th row of the $(r,p+q)$-rectangle. Such an assignment is called a $(p,q,r)$-design if $\{x_i : 1\le i\le p\}=\{c_1\}$ for some constant $c_1$, $\{y_j : 1\le j\le q\}=\{c_2\}$ for some constant $c_2$, and $\{z_k : 1\le k\le r\}=\{c_3\}$ for some constant $c_3$. A $(p,q,r)$-board that admits a $(p,q,r)$-design is called (1) Cartesian tri-magic if $c_1$, $c_2$ and $c_3$ are all distinct; (2) Cartesian bi-magic if $c_1$, $c_2$ and $c_3$ assume exactly 2 distinct values; (3) Cartesian magic if $c_1 = c_2 = c_3$ (which is equivalent to supermagic labeling of $K(p,q,r)$). Thus, Cartesian magicness is a generalization of magic rectangles into 3-dimensional space. In this paper, we study the Cartesian magicness of various $(p,q,r)$-board by matrix approach involving magic squares or rectangles. In Section~2, we obtained various sufficient conditions for  $(p,q,r)$-boards to admit a Cartesian tri-magic design.  In Sections~3 and~4, we obtained many necessary and (or) sufficient conditions for various $(p,q,r)$-boards to admit (or not admit) a Cartesian bi-magic and magic design. In particular, it is known that $K(p,p,p)$ is supermagic and thus every $(p,p,p)$-board is Cartesian magic. We gave a short and simpler proof that every $(p,p,p)$-board is Cartesian magic.\\ 


\end{abstract}

\section{Introduction}\label{intro}\indent

For positive integers $p_i, 1\le i\le k, k\ge 2$, the $k$-tuple $(p_1, p_2, \ldots, p_k)$ is called a {\it $(p_1, p_2, \ldots, p_k)$-board}, or a {\it generalized plane}, in $k$-space that is formed by ${k\choose 2}$ rectangles $P_iP_j$ ($1\le i < j\le k$) of size $p_i\times p_j$. Abusing the notation, we also let $P_iP_j$ denote a matrix of size $p_i\times p_j$, where $P_iP_j$ is an entry of a block matrix $M$ as shown below
\begin{equation}\label{eq-M} M=\begin{pmatrix}\bigstar & P_1P_2 & P_1P_3 & \cdots & \cdots & P_1P_k \\
 (P_1P_2)^T & \bigstar & P_2P_3 & \cdots & \cdots & P_2P_k \\
 (P_1P_3)^T & (P_2P_3)^T & \bigstar & \cdots & \cdots & P_3P_k \\
 \vdots & \vdots & \vdots & \ddots & \cdots & \vdots \\
  (P_1P_{k-1})^T & (P_2P_{k-1})^T &(P_3P_{k-1})^T & \cdots & \ddots & P_{k-1}P_{k}\\
 (P_1P_k)^T & (P_2P_k)^T &(P_3P_k)^T & \cdots &  (P_{k-1}P_k)^T & \bigstar
\end{pmatrix}\end{equation}
such that the $r$-th row of $M$, denoted $M_r$ $(1\le r\le k)$, is a submatrix of size $p_r \times (p_1 + \cdots + p_k)$. For $1\le d\le k$, we say a $(p_1, \ldots, p_k)$-board is {\it $(k,d)$-magic} if
\begin{enumerate}[(i)]
  \item the row sum of all the entries of each row of $M_r$ is a constant $c_r$, and
  \item $\{c_1,c_2,\ldots, c_k\}$ has exactly $d$ distinct elements.
\end{enumerate}
We also say that a $(k,d)$-magic $(p_1, \ldots, p_k)$-board admits a {\it $(k,d)$-design}. Thus, a (2,2)-magic $(p,q)$-board is what has been known as a magic rectangle while a (2,1)-magic $(p,q)$-board is what has been known as a magic square. We shall say a (3,3)-magic, a (3,2)-magic and a (3,1)-magic  $(p,q,r)$-board is {\it Cartesian tri-magic}, {\it Cartesian bi-magic} and {\it Cartesian magic} respectively. In this paper, we determine Cartesian magicness of $(p,q,r)$-boards by matrix approach involving magic squares or rectangles.

For $a,b\in \Z$ and $a\le b$, we use $[a,b]$ to denote the set of integers from $a$ to $b$. Let $S$ be the set of the $pq+pr+qr$ squares of a $(p,q,r)$-board. Consider a bijection $f : S \to [1,pq+pr+qr]$. 
For convenience of presentation, throughout this paper, we let $PQ$, $PR$, and $QR$ be the images of $(p,q)$-, $(p,r)$-, and $(q,r)$-rectangles under $f$ in matrix form, respectively. Hence, $PQ$, $PR$ and $QR$ are matrices of size $p\times q$, $p\times r$ and $q\times r$, respectively. 

Let $G = (V,E)$ be a graph containing $p$ vertices and $q$ edges. If there exists a bijection
$f:E \rightarrow [1, q]$ such that the map $\displaystyle f^+(u) =
\sum_{uv \in E}\hspace{-1mm} f(uv)$ induces a constant map from $V$ to $\Z$,
then $G$ is called {\it supermagic} and $f$ is called a {\it supermagic labeling} of $G$ \cite{Ste1966, Ste1967}.

 A {\it labeling matrix} for a labeling $f$ of $G$ is a matrix whose rows and columns are named by the vertices of $G$ and the $(u,v)$-entry is $f(uv)$ if $uv \in E$, and is $*$ otherwise. Sometimes, we call this matrix a {\it labeling matrix of $G$}. In other words, suppose $A$ is an adjacency matrix of $G$ and $f$ is a labeling of $G$, then a labeling matrix for $f$ is obtained from $A=(a_{u,v})$ by replacing $a_{u,v}$ by $f(uv)$ if $a_{u,v}=1$ and by $*$ if $a_{u,v}=0$. This concept was first introduced by Shiu, {\it et al.} in \cite{Shiu1998}. Moreover, if $f$ is a supermagic labeling, then a labeling matrix of $f$ is called a {\it supermagic labeling matrix} of $G$ \cite{Shiu2002}. Thus, a simple $(p,q)$-graph $G=(V,E)$ is supermagic if and only if there exists a bijection $f:E \rightarrow [1, q]$ such that the row sums (as well as the column sums) of the labeling matrix for $f$ are the same. For purposes of these sums, entries labeled with $*$ will be treated as $0$. It is easy to see that $K(p,q,r)$ is supermagic if and only if the $(p,q,r)$-board is Cartesian magic.

Note that the block matrix $M$ in \eqref{eq-M} is a labeling matrix of the complete $k$-partite graph $K(p_1, \dots, p_k)$. In particular, consider the complete tripartite graph $K(p,q,r)$. Suppose $f$ is an edge-labeling of $K(p,q,r)$. According to the vertex-list $\{x_1, \dots, x_p, y_1, \dots, y_q, z_1, \dots, z_r\}$, the labeling matrix of $f$ is
$$M=\begin{pmatrix}\bigstar & PQ & PR \\
 (PQ)^T & \bigstar & QR\\
 (PR)^T & (QR)^T & \bigstar \end{pmatrix},$$
where $PQ$, $PR$ and $QR$ are defined before, each $\bigstar$ is a certain size matrix whose entries are $*$. For convenience, we use $QP$, $RP$ and $RQ$ to denote $(PQ)^T$, $(PR)^T$ and $(QR)^T$, respectively.


\section{Cartesian tri-magic}\indent

In this section, we will make use of the existence of magic rectangles. From \cite{Chai, Reyes}, we know that a $h\times k$ magic rectangle exists when $h,k\ge 2$, $h\equiv k\pmod{2}$ and $(h,k)\ne (2,2)$.

\begin{theorem}\label{thm-1pqreven1} Suppose $3\le p < q < r$, where $p$, $q$ are odd and $r$ is even. The $(p,q,r)$-board is Cartesian tri-magic.  \end{theorem}

\begin{proof}  Fill the $(p + q) \times r$ rectangle with integers in $[1, (p+q)r]$ and the $p \times q$ rectangle  with integers in $[(p+q)r + 1, pq + pr + qr]$ to form two magic rectangles.  Thus, $c_1 = (pr^2+qr^2+pq^2+q+r)/2 + pqr+q^2r$, $c_2 = (p^2q+pr^2+qr^2+p+r)/2 + pqr+p^2r$, and $c_3 = (p^2r+q^2r+p+q)/2 + pqr$. Observe that
\begin{align*}
2(c_1 - c_2) & = (q-p)(pq+1+2(p+q)r),\\
2(c_2 - c_3) & = pr(r+p)+(qr+1)(r-q)+p^2q.
\end{align*}
Clearly, $c_1>c_2>c_3$. Hence, the theorem holds.  \end{proof}

\begin{corollary} Suppose $1\le p < r$, where $p$ is odd and $r$ is even. The $(p,p,r)$-board is Cartesian bi-magic.
\end{corollary}
\begin{proof}
When $p=1$, the corollary follows from Theorem~\ref{thm-11r}. For $p\ge 3$, we make use of the equations obtained in  Theorem~\ref{thm-1pqreven1} and let $p=q$. \end{proof}

\begin{theorem}\label{thm-1pqreven2} If $3\le p < q < r$, where $q$ is even, and $p$ and $r$ are odd, then the $(p,q,r)$-board is Cartesian tri-magic.  \end{theorem}

\begin{proof}
Fill the $p \times r$ rectangle with integers in $[1, pr]$ and $(p + r) \times q$ rectangle with integers in $[pr+1, pq+pr+qr]$ to form two magic rectangles. We have
\begin{align*}
2(c_2 - c_1) & = p^2q-pq^2-q^2r+qr^2+pr^2+2p^2r+p-q\\&=p(r^2-p^2)+qr(r-q)+2p^2r+(p^2-1)q+p>0,\\
2(c_1 - c_3) & = pr^2-p^2r+r-p=(r-p)(pr+1)>0.
\end{align*}
Thus $c_2>c_1>c_3$. Hence the theorem holds.
\end{proof}

\begin{theorem}\label{thm-1pqreven3} If  $2\le p < q < r$, where $p$ is even, and $q$ and $r$ are odd, then the $(p,q,r)$-board is Cartesian tri-magic.  \end{theorem}

\begin{proof} Fill the $q\times r$ rectangle with integers in $[1, qr]$ and
the $p \times (q+r)$ rectangle with integers in $[qr+1, pq+pr+qr]$ to form two magic rectangles.  Thus, $2c_1 = 2q^2r+2qr^2+pq^2+pr^2+q+r+2pqr$, $2c_2 = qr^2+p^2q+p^2r+r+p+2pqr$, and $2c_3 = p^2q+p^2r+q^2r+q+p+2pqr$. Now
\begin{align*}
2(c_2 - c_3) & = (r-q)(qr+1),\\
2(c_1 - c_2) & = pr(r-p)+(pq+1)(q-p)+2q^2r+qr^2.
\end{align*}
Clearly, $c_1>c_2>c_3$. Hence, the theorem holds.  \end{proof}

\begin{corollary} Suppose $2 \le p < r$, where $p$ is even and $r$ is odd. The $(p,r,r)$-board is Cartesian bi-magic.
\end{corollary}

\begin{theorem}\label{thm-pqrmagic1} Suppose $2\le  p \le q \le r$, where $p$, $q$ and $r$ have the same parity and $(p, q) \not= (2, 2)$. Then $(p,q,r)$-board is Cartesian tri-magic. \end{theorem}

\begin{proof} Fill the $p \times r$ rectangle with integers in $[1, pr]$, the $p \times q$ rectangle with integers in $[pr + 1, pr + pq]$, and the $q \times r$ rectangle with integers in $[pr + pq + 1, pr +  pq+qr]$ to form three magic rectangles.  Now, $2c_1 = 2pqr+pq^2+pr^2+r+q$, $2c_2 = 2pqr+2p^2r+2pr^2+p^2q+qr^2+r+p$, and $2c_3 = 2pqr+2pq^2+p^2r+q^2r+q+p$.
Therefore,
\begin{align}
2(c_2 - c_3) & = p^2r+2pr^2+p^2q+qr^2-2pq^2-q^2r+r-q\nonumber\\ &=2p(r^2-q^2)+p^2(r+q)+(r-q)(qr+1)>0,\nonumber\\
2(c_2 - c_1) & = 2p^2r+p^2q+qr^2+pr^2-pq^2+p-q=2p^2r+p^2q+pr^2+p+q(r^2-pq-1)>0,\nonumber\\
2(c_3 - c_1) & = pq^2+p^2r+q^2r -pr^2-r+p.\label{eq-pqrsameA}
\end{align}


Fill the $p \times r$ rectangle with integers in $[1, pr]$, the $q \times r$ rectangle with integers in $[pr + 1, pr + qr]$, and the $p \times q$ rectangle with integers in $[pr + qr + 1, pr + qr+pq]$ to form three magic rectangles.  Now, $2c_1 = 2pqr+pq^2+pr^2+2q^2r+r+q$, $2c_2 = 2pqr+2p^2r+2pr^2+p^2q+qr^2+r+p$, and $2c_3 = 2pqr+p^2r+q^2r+q+p$.
Therefore,
\begin{align}
2(c_1 - c_3) & = pr^2-p^2r+q^2r+pq^2+r-p=(r-p)(pr+1)+q^2(p+r)>0,\nonumber\\
2(c_2 - c_3) & = p^2r+2pr^2+p^2q+qr^2-q^2r+r-q=p^2(r+q)+2pr^2+(r-q)(qr+1)>0,\nonumber\\
2(c_2 - c_1) & = pr^2+qr^2+2p^2r-2q^2r+p^2q-pq^2-q+p.\label{eq-pqrsame0}
\end{align}
The sum of \eqref{eq-pqrsameA} and \eqref{eq-pqrsame0} is
\begin{equation*}
3p^2r-q^2r+qr^2+p^2q-r-q+2p = qr(r-q)+q(p^2-1)+r(3p^2-1)+2p>0.
\end{equation*}
So at least one of \eqref{eq-pqrsameA} and \eqref{eq-pqrsame0} is positive. Hence we have the theorem.
\end{proof}

For $p=q=2$, we have the following.

\begin{theorem}\label{thm-K(2r2)} For all $r\ge 1$, the $(2,2,r)$-board is Cartesian tri-magic. \end{theorem}

\begin{proof}
For $r=1$, a labeling matrix for $(1,2,2)$ is:
\begin{equation*}
\left(\begin{array}{c|cc|cc}
* & 6 & 8 & 1 & 5\\\hline
6 & * & * & 7 & 2\\
8 & * & * & 3 & 4\\\hline
1 & 7 & 3 & * & *\\
5 & 2 & 4 & * & *
\end{array}\right)\left(\begin{array}{c}
20\\\hline
15\\15\\\hline
11\\11
\end{array}\right)
\end{equation*}
The right column contains the row sums of the left matrix.
\\

 For $r=2$, consider\\
{\footnotesize
\begin{equation*}
PQ=\begin{array}{|c|c|}
\hline
2 & 4 \\
\hline
1 & 3 \\
\hline
\end{array}\qquad
PR=\begin{array}{|c|c|}
\hline
7 & 5 \\
\hline
6 & 8 \\
\hline
\end{array}\qquad
QR=\begin{array}{|c|c|}
\hline
11 & 12 \\
\hline
10 & 9 \\
\hline
\end{array}
\end{equation*}}

 Clearly, we get a Cartesian tri-magic design with $c_1= 18$, $c_2=26$, and $c_3=34$.  \\

Now assume $r\ge 3$. For $r\equiv 0 \pmod{4}$, consider\\
\vskip-0.5cm
 \begin{table}[H]
      \centering
\footnotesize
$PQ=$
\begin{tabular}{|c|c|}
\hline
$4r+1$ & $4r+4$ \\
\hline
$4r+3$ & $4r+2$ \\
\hline
\end{tabular}
    \end{table}
\vskip-0.5cm

\begin{table}[H]
\footnotesize
$PR=$
     \centering
\setlength{\extrarowheight}{1pt}
\setlength{\tabcolsep}{2.5pt}
\begin{tabular}{|c|c|c|c"c"c|c|c|c"c|c|c|c|}
\hline
1 & $2r-1$ & $2r-2$ & 4 & $\cdots$ & $r-7$ & $r+7$ & $r+6$ & $r-4$ &  $r-3$ & $r+3$ & $r+2$ & $r$ \\
\hline
$2r$ & 2 & 3 & $2r-3$ & $\cdots$ & $r+8$ & $r-6$ &  $r-5$ & $r+5$  & $r+4$ & $r-2$ & $r-1$ & $r+1$ \\
\hline
\end{tabular}
\end{table}

 \begin{table}[H]
\vskip-0.5cm
\footnotesize
$QR=$
     \centering
\setlength{\extrarowheight}{1pt}
\setlength{\tabcolsep}{1.8pt}
\begin{tabular}{|c|c|c|c"c"c|c|c|c"c|c|c|c|}
\hline
$2r+1$ & $4r-1$ & $4r-2$ & $2r+4$ & $\cdots$ & $3r-7$ & $3r+7$ &  $3r+6$ & $3r-4$  & $3r-3$ & $3r+3$ & $3r+2$ & $3r+1$ \\
\hline
$4r$ & $2r+2$ & $2r+3$ & $4r-3$ & $\cdots$ & $3r+8$ & $3r-6$ & $3r-5$ & $3r+5$ &  $3r+4$ & $3r-2$ & $3r-1$ & $3r$ \\
\hline
\end{tabular}
\end{table}

 Clearly, we get a Cartesian tri-magic design with $c_1= r^2+17r/2+5$, $c_2=3r^2+17r/2+5$, and $c_3=8r+2$.\\ 

 For $r \equiv 1 \pmod{4}$, consider \\
 \begin{table}[H]
    \centering
\vskip-0.6cm
\footnotesize
$PQ=$
\begin{tabular}{|c|c|}
\hline
$4r+3$ & $4r+2$ \\
\hline
$4r+1$ & $4r+4$ \\
\hline
\end{tabular}
    \end{table}
\vskip-0.5cm
\begin{table}[H]
\footnotesize
$PR=$
      \centering
\setlength{\extrarowheight}{1pt}
\setlength{\tabcolsep}{2.5pt}
\begin{tabular}{|c|c|c|c"c"c|c|c|c"c|c|c|c|c|}
\hline
1 & $2r-1$ & $2r-2$ & 4 & $\cdots$ & $r-8$ & $r+8$ & $r+7$ & $r-5$ &  $r-4$ & $r+4$ & $r-2$ & $r+2$ & $3r$ \\
\hline
$2r$ & 2 & 3 & $2r-3$ & $\cdots$ & $r+9$ & $r-7$ &  $r-6$ & $r+6$  & $r+5$ & $r-3$ & $3r-2$ & $r-1$ & $r+1$\\
\hline
\end{tabular}
\end{table}

\begin{table}[H]
\vskip-0.5cm
\footnotesize
$QR=$
      \centering
\setlength{\extrarowheight}{1pt}
\setlength{\tabcolsep}{0.85pt}
\begin{tabular}{|c|c|c|c"c"c|c|c|c"c|c|c|c|c|}
\hline
$2r+1$ & $4r-1$ & $4r-2$ & $2r+4$ & $\cdots$ & $3r-8$ & $3r+8$ &  $3r+7$ & $3r-5$  & $3r-4$ & $3r+4$ & $r+3$ & $3r+2$ & $3r+1$ \\
\hline
$4r$ & $2r+2$ & $2r+3$ & $4r-3$ & $\cdots$ & $3r+9$ & $3r-7$ & $3r-6$ & $3r+6$ &  $3r+5$ & $3r-3$ & $3r+3$ & $3r-1$ & $r$\\
\hline
\end{tabular}
\end{table}

 Clearly, we get a Cartesian tri-magic design with $c_1=r^2+ (21r+5)/2$, $c_2=3r^2+(13r+15)/2$, and $c_3=8r+2$.\\

 For $r \equiv 2 \pmod{4}$, consider\\

\vskip-0.5cm
\begin{table}[H]
      \centering
\footnotesize
$PQ=$
\begin{tabular}{|c|c|}
\hline
$4r+2$ & $4r+4$ \\
\hline
$4r+1$ & $4r+3$ \\
\hline
\end{tabular}
    \end{table}

\vskip-0.5cm
\begin{table}[H]
\footnotesize
$PR=$
      \centering
\setlength{\extrarowheight}{1pt}
\setlength{\tabcolsep}{2.5pt}
\begin{tabular}{|c|c|c|c"c"c|c|c|c"c|c|}
\hline
1 & $2r-1$ & $2r-2$ & 4 & $\cdots$ & $r-5$ & $r+5$ & $r+4$ & $r-2$ &  $r-1$ & $r+1$  \\
\hline
$2r$ & 2 & 3 & $2r-3$ & $\cdots$ & $r+6$ & $r-4$ &  $r-3$ & $r+3$  & $r+2$ & $r$ \\
\hline
\end{tabular}
     \centering
\end{table}
\begin{table}[H]
\vskip-0.5cm
\footnotesize
$QR=$
      \centering
\setlength{\extrarowheight}{1pt}
\setlength{\tabcolsep}{2.3pt}
\begin{tabular}{|c|c|c|c"c"c|c|c|c"c|c|}
\hline
$2r+1$ & $4r-1$ & $4r-2$ & $2r+4$ & $\cdots$ & $3r-5$ & $3r+5$ &  $3r+4$ & $3r-2$  & $3r+2$ & $3r+1$  \\
\hline
$4r$ & $2r+2$ & $2r+3$ & $4r-3$ & $\cdots$ & $3r+6$ & $3r-4$ & $3r-3$ & $3r+3$ &  $3r-1$ & $3r$ \\
\hline
\end{tabular}
\end{table}
 Clearly, we get a Cartesian tri-magic design with $c_1=r^2+17r/2+5$, $c_2=3r^2+17r /2+5$, and $c_3=8r+2$.\\

 Finally for $r \equiv 3 \pmod{4}$, consider\\

\begin{table}[H]
\vskip-0.5cm
\footnotesize
$PQ=$
      \centering
\begin{tabular}{|c|c|}
\hline
$4r+1$ & $4r+3$ \\
\hline
$4r+2$ & $4r+4$ \\
\hline
\end{tabular}
    \end{table}
\begin{table}[H]

\vskip-0.5cm

\footnotesize
$PR=$
      \centering
\setlength{\extrarowheight}{1pt}
\setlength{\tabcolsep}{2.5pt}
\begin{tabular}{|c|c|c|c"c"c|c|c|c"c|c|c|}
\hline
1 & $2r-1$ & $2r-2$ & 4 & $\cdots$ & $r-6$ & $r+6$ & $r+5$ & $r-3$ &  $r-2$ & $r+2$ & $3r$  \\
\hline
$2r$ & 2 & 3 & $2r-3$ & $\cdots$ & $r+7$ & $r-5$ &  $r-4$ & $r+4$  & $3r-2$ & $r-1$ & $r+1$  \\
\hline
\end{tabular}
    \centering
\end{table}
%
%
\begin{table}[H]
\vskip-0.5cm
\footnotesize
$QR=$
     \centering
\setlength{\extrarowheight}{1pt}
\setlength{\tabcolsep}{2.2pt}
\begin{tabular}{|c|c|c|c"c"c|c|c|c"c|c|c|}
\hline
$2r+1$ & $4r-1$ & $4r-2$ & $2r+4$ & $\cdots$ & $3r-6$ & $3r+6$ &  $3r+5$ & $3r-3$  & $r+3$ & $3r+2$ & $3r+1$ \\
\hline
$4r$ & $2r+2$ & $2r+3$ & $4r-3$ & $\cdots$ & $3r+7$ & $3r-5$ & $3r-4$ & $3r+4$ &  $3r+3$ & $3r-1$ & $r$\\
\hline
\end{tabular}
\end{table}

 Clearly, we get a Cartesian tri-magic design with $c_1=r^2+(21r+5)/2$, $c_2=3r^2+(13r+15)/2$, and $c_3=8r+2$.   \end{proof} 

\begin{corollary} The $(p,p,p)$-board is Cartesian tri-magic for all $p\ge 1$. \end{corollary}

We now consider the case $p=1$. We first introduce some notation about matrices.\\

Let $m,n$ be two positive integers. For convenience, we use $M_{m,n}$ to denote the set of $m\times n$ matrices over $\Z$. For any matrix $M\in M_{m,n}$, $r_i(M)$ and $c_j(M)$ denote the $i$-th row sum and the $j$-th column sum of $M$, respectively.\\

We want to assign the integers in $[1, q+r+qr]$ to matrices $PR\in M_{1,r}$, $QR\in M_{q,r}$ and $QP=(PQ)^T\in M_{q,1}$ such that the matrix
\begin{equation*}
M=\begin{pmatrix}
* & PR\\
QP & QR\end{pmatrix}\end{equation*}
has the following properties:
\begin{enumerate}[P.1]
\item Each integers in $[1, q+r+qr]$ appears once.
\item $r_i(M)$ is a constant not equal to $r_1(M)+c_1(M)$, $2\le i\le q+1$.
\item $c_j(M)$ is a constant not equal to $r_i(M)$ or $r_1(M)+c_1(M)$, $2\le j\le r+1$.
\end{enumerate}
Such a matrix $M$ is called a {\it Cartesian labeling matrix} of the $(1, q, r)$-board (or the graph $K(1,q,r)$.)

\begin{theorem}  The $(1,1,r)$-board is Cartesian tri-magic. \end{theorem}

\begin{proof} A Cartesian labeling matrix of the $(1,1,r)$-board is
\[\left(\begin{array}{c|cccc}
* & 1 & 2 & \cdots & r\\\hline
2r+1 & 2r & 2r-1 & \cdots & r+1\end{array}\right)
\]
Clearly, we get a Cartesian tri-magic design with $c_1= (r^2+5r+2)/2, c_2 = (3r^2+5r+2)/2$ and $c_3 = 2r+1$.
\end{proof}


Note that the $(1,1,2)$-board also admits a different Cartesian labeling matrix
\[ \left(\begin{array}{c|cc}
* & 2 & 5\\\hline
3 & 4 & 1
\end{array}\right)\]
with  $c_1=10$, $c_2=8$, and $c_3=6$ respectively. 

\begin{theorem}\label{thm-1qqbimagic} Suppose $q\equiv r\pmod{2}$ and $q\ge 2$. The $(1,q,r)$-board is Cartesian tri-magic if $q<r$; and is Cartesian bi-magic if $q=r$.  \end{theorem}
\begin{proof}
Let $A$ be a $(q+1)\times (r+1)$ magic rectangle. Exchanging columns and exchanging rows if necessary, we may assume that $(q+1)(r+1)$ is put at the $(1,1)$-entry of $A$. Now let $PR$ be the $1\times r$ matrix obtained from the first row of $A$ by deleting the $(1,1)$-entry; let $QP$ be the $r\times 1$ matrix obtained from the first column of $A$ by deleting the $(1,1)$-entry; let $QR$ be the $r\times r$ matrix obtained from $A$ by deleting the first row and the first column.

It is easy to check that $c_1=\frac{(q+r+2)[(q+1)(r+1)+1]}{2}-2(q+1)(r+1)$, $c_2=\frac{(r+1)[(q+1)(r+1)+1]}{2}$, and $c_3=\frac{(q+1)[(q+1)(r+1)+1]}{2}$. Also, $c_1>c_2>c_3$ if $q<r$; and $c_1>c_2=c_3$ if $q=r$.
\end{proof}

Suppose $q=2s+1$ and $r=2k$, where $k>s\ge 1$. We assign the integers in $[1, 4sk+4k+2s+1]$ to form a matrix $M$ satisfying the properties P.1-P.3.

Let $\alpha=\begin{pmatrix}1 & 2& \cdots & k\end{pmatrix}$ and $\beta=\begin{pmatrix}k & k-1& \cdots & 1\end{pmatrix}$ be row vectors in $M_{1,k}$. Let $J_{m,n}$ be the $m\times n$ matrix whose entries are $1$. \\

Let $A=\left(\begin{array}{c|c}
\alpha +[2s+1]J_{1,k}&
\alpha+[2s+1+k]J_{1,k}\\
\beta +[2s+1+3k]J_{1,k} & \beta+[2s+1+2k]J_{1,k}\\
\alpha +[2s+1+4k]J_{1,k}&
\alpha+[2s+1+5k]J_{1,k}\\
\vdots & \vdots\\
\beta +[2s+1+(4s-1)k]J_{1,k} & \beta+[2s+1+(4s-2)k]J_{1,k}\\
\alpha +[2s+1+4sk]J_{1,k}&
\alpha+[2s+1+(4s+1)k]J_{1,k}
\end{array}\right) \in M_{2s+1,2k}$. We separate $A$ into two parts, left and right. Now reverse the rows of the right part of $A$ from top to bottom:\\

$B=\left(\begin{array}{c|c}
\alpha +[2s+1]J_{1,k}& \alpha+[2s+1+(4s+1)k]J_{1,k}\\
\beta +[2s+1+3k]J_{1,k} &  \beta+[2s+1+(4s-2)k]J_{1,k}\\
\alpha +[2s+1+4k]J_{1,k}& \alpha+[2s+1+(4s-3)k]J_{1,k}\\
\vdots & \vdots\\
\beta +[2s+1+(4s-1)k]J_{1,k} &\beta+[2s+1+2k]J_{1,k}\\
\alpha +[2s+1+4sk]J_{1,k}&
\alpha+[2s+1+k]J_{1,k}
\end{array}\right)$.

We insert $\left(\begin{array}{@{}c|c@{}}\beta +[2s+1+(4s+3)k]J_{1,k} &\beta+[2s+1+(4s+2)k]J_{1,k}\end{array}\right)$ to $B$ as the first row. So we get
$C=\left(\begin{array}{c|c}
\beta +[2s+1+(4s+3)k]J_{1,k} &\beta+[2s+1+(4s+2)k]J_{1,k}\\\hline
\alpha +[2s+1]J_{1,k}& \alpha+[2s+1+(4s+1)k]J_{1,k}\\
\beta +[2s+1+3k]J_{1,k} &  \beta+[2s+1+(4s-2)k]J_{1,k}\\
\alpha +[2s+1+4k]J_{1,k}& \alpha+[2s+1+(4s-3)k]J_{1,k}\\
\vdots & \vdots\\
\beta +[2s+1+(4s-1)k]J_{1,k} &\beta+[2s+1+2k]J_{1,k}\\
\alpha +[2s+1+4sk]J_{1,k}&
\alpha+[2s+1+k]J_{1,k}
\end{array}\right)\in M_{2s+2,2k}$.\\

Each column sum of $C$ is $(s+1)(4sk+4s+4k+3)$.
Each row sum (except the 1st row) of $C$ is $k(4sk+4s+2k+3)$ and $r_1(C)=k(8sk+4s+6k+3)$. The set of remaining integers is $[1,2s+1]$ which will form the column matrix $QP$.

It is easy to see that the difference between the $(2i+1)$-st and the $(2i+2)$-nd rows of $C$ is\\ $\begin{pmatrix}-1 & -3 & \cdots & -(2k-3) & -(2k-1) & 2k-1 & 2k-3 & \cdots & 3 & 1\end{pmatrix}$ for $1\le i\le s$.

We let $QP=\left(\begin{array}{@{}c|cc|cc|cc|cc@{}}s+1 & s+2 & s & s+3 & s-1& \cdots & \cdots & 2s+1 & 1\end{array}\right)^T\in M_{2s+1,1}$.

Now let
\[
N=\left(\begin{array}{c|c|c}
* & \beta +[2s+1+(4s+3)k]J_{1,k} &\beta+[2s+1+(4s+2)k]J_{1,k}\\\hline
s+1 & \alpha +[2s+1]J_{1,k}& \alpha+[2s+1+(4s+1)k]J_{1,k}\\
s+2 & \beta +[2s+1+3k]J_{1,k} &  \beta+[2s+1+(4s-2)k]J_{1,k}\\
s & \alpha +[2s+1+4k]J_{1,k}& \alpha+[2s+1+(4s-3)k]J_{1,k}\\
\vdots & \vdots & \vdots\\
\vdots & \vdots & \vdots\\
2s+1 & \beta +[2s+1+(4s-1)k]J_{1,k} &\beta+[2s+1+2k]J_{1,k}\\
1 & \alpha +[2s+1+4sk]J_{1,k} & \alpha+[2s+1+k]J_{1,k}
\end{array}\right)
\]
Here $r_{2i+1}(N)-r_{2i+2}(N)=2i$, $1\le i\le s$. For odd $i$, we swap the $(2i+1, 2k+1-i)$-entry with the $(2i+2, 2k+1-i)$-entry. For even $i$, we swap the $(2i+1, 2k+1-i)$-entry with the $(2i+2, 2k+1-i)$-entry of $N$ and swap the $(2i+1, 2)$-entry with the $(2i+2, 2)$-entry of $N$. Note that, they work since $1\le i\le s<k$. The resulting matrix is the required matrix $M\in M_{2s+2, 2k+1}$. Note that $c_2=r_i(M)=k(4sk+4s+2k+3)+s+1$ for $2\le i\le 2s+2$, $c_3=c_j(M)=(s+1)(4sk+4s+4k+3)$ for $2\le j\le 2k+1$ and $c_1=r_1(M)+c_1(M)=k(8sk+4s+6k+3)+(s+1)(2s+1)$.

\begin{rem}{\rm In the above construction, we use integers in $[2s+2, 4sk+4k+2s+1]$ to form the matrix $C$. We may use integers in $[1, 4sk+4k]$ to form a new matrix $C'$, namely $C'=C-(2s+1)J_{2s+2, 2k}$. The remaining integers of $[4sk+4k+1, 4sk+4k+2s+1]$ form the new matrix $PQ'$, namely $PQ'=PQ+(4sk+4k)J_{2s+1,1}$. By the same procedure as above, we have a new matrix $M'$ with $c_2'=r_i(M')=r_i(M)+2k=k(4sk+4s+2k+5)+s+1$ for $2\le i\le 2s+2$, $c_3'=c_j(M')=c_j(M)-(2s+2)(2s+1)=(s+1)(4sk+4k+1)$ for $2\le j\le 2k+1$ and $c_1'=r_1(M')+c_1(M')=r_1(M)+c_1(M)-(2s+1)2k+(2s+1)(4sk+4k)=
k(8sk+4s+6k+3)+(2s+1)(4sk+2k)$. So, if $c_2=c_3$ in the above discussion, then we may change the arrangement $M$ to $M'$ to obtain a Cartesian tri-magic labeling for the $(1, 2s+1, 2k)$-board.
}\end{rem}

Thus we have

\begin{theorem}\label{thm-1oqer} Suppose $q\ge 3$ is odd and $r$ is even. The $(1,q,r)$-board is Cartesian tri-magic.  \end{theorem}

\begin{example}{\rm $(1,5,8)$-board
\begin{equation*}
\left(\begin{array}{*{4}c|*{4}c}
53 & 52 & 51 & 50 & 49 & 48 & 47 & 46\\\hline
6 & 7 & 8 & 9 & 10 & 11 & 12 & 13\\
21 & 20 & 19 & 18 & 17 & 16 & 15 & 14\\
22 & 23 & 24 & 25 & 26 & 27 & 28 & 29\\
37 & 36 & 35 & 34 & 33 & 32 & 31 & 30\\
38 & 39 & 40 & 41 & 42 & 43 & 44 & 45
\end{array}\right)\rightarrow C=
\left(\begin{array}{*{4}c|*{4}c}
53 & 52 & 51 & 50 & 49 & 48 & 47 & 46\\\hline
6 & 7 & 8 & 9 & 42 & 43 & 44 & 45\\
21 & 20 & 19 & 18 & 33 & 32 & 31 & 30\\
22 & 23 & 24 & 25 & 26 & 27 & 28 & 29\\
37 & 36 & 35 & 34 & 17 & 16 & 15 & 14\\
38 & 39 & 40 & 41 & 10 & 11 & 12 & 13
\end{array}\right)
\end{equation*}
The first row is the matrix $PR$ and the last 5 rows form the matrix $QR$.

Now each column sum is 177, each row sum of $QR$ is 204. But we have to put 1,2,3,4,5 into the matrix $PQ$ (or $QP$). The average of these numbers is 3. So we have to make each row sum of the augmented matrix $(QP|QR)$ to be 207. Thus we put these numbers into $QP$ as follows:
\begin{equation*}
N=
\left(\begin{array}{c|*{4}c|*{4}c}
* & 53 & 52 & 51 & 50 & 49 & 48 & 47 & 46\\\hline
3 & 6 & 7 & 8 & 9 & 42 & 43 & 44 & 45\\
4 & 21 & 20 & 19 & 18 & 33 & 32 & 31 & 30\\
2 & 22 & 23 & 24 & 25 & 26 & 27 & 28 & 29\\
5 & 37 & 36 & 35 & 34 & 17 & 16 & 15 & 14\\
1 & 38 & 39 & 40 & 41 & 10 & 11 & 12 & 13
\end{array}\right)
\end{equation*}
Now the row sums of $QR$ are 207, 208, 206, 209 and 205. So we must swap some entries of $QR$. We will pair up  rows of $QR$ and $QP$, namely 2nd and 3rd, 4th and 5th. The 2nd row sum is greater than the 3rd row sum by 2; and the 4th row sum is greater than the 5th row sum by 4. In 2nd and 3rd row of $QR$, there are two entries at the same column whose difference is 1 (namely 29 and 30); two entries at the same column with difference $-1$ (namely 37 and 38) and two entries at the same column with difference $+3$ (namely 15 and 12). So, swapping these pairs of integers we get
\begin{equation*}
M=\left(\begin{array}{c|*{4}c|*{4}c}
* & 53 & 52 & 51 & 50 & 49 & 48 & 47 & 46\\\hline
3 & 6 & 7 & 8 & 9 & 42 & 43 & 44 & 45\\
4 & 21 & 20 & 19 & 18 & 33 & 32 & 31 & \round{29}\\
2 & 22 & 23 & 24 & 25 & 26 & 27 & 28 & \round{30}\\
5 & \round{38} & 36 & 35 & 34 & 17 & 16 & \round{12} & 14\\
1 & \round{37} & 39 & 40 & 41 & 10 & 11 & \round{15} & 13
\end{array}\right)
\end{equation*}
Now
$c_1=411$, $c_2=207$ and $c_3=177$.

Or
\begin{equation*}
M'=
\left(\begin{array}{c|*{4}c|*{4}c}
* & 48	& 47 & 46 &	45 & 44 & 43 & 42 &	41\\\hline
51 & 1 & 2 & 3 & 4 & 37 & 38 & 39 & 40\\
52 & 16 & 15 & 14 & 13 & 28 & 27 & 26 & \round{24}\\
50 & 17 & 18 & 19 &	20 & 21 & 22 & 23 & \round{25}\\
53 & \round{33}	& 31 & 30 & 29 & 12 & 11 & \round{7} & 9\\
49 & \round{32} & 34 & 35 & 36 & 5 & 6 & \round{10} & 8 \end{array}\right)\end{equation*}
Now
$c_1'=611$, $c_2'=215$ and $c_3'=147$.
}
\end{example}

\medskip

Suppose $q=2s$ and $r=2k-1$, where $k>s\ge 1$. We want to assign the integers in $[1, 4sk+2k-1]$ to form a matrix $M$ satisfying the properties P.1-P.3.

When $s=1$, we have the following.\\


For $r=3$, consider the Cartesian labeling matrix
\[ \left(\begin{array}{c|ccc}
* & 2 & 4 & 6\\\hline
8 & 5 & 11 & 3\\
10 & 9 & 1 & 7
\end{array}\right).\]

 Clearly, we get a Cartesian tri-magic design with $c_1=30$, $c_2=27$, and $c_3=16$.\\ 


Now for $r \equiv 1 \pmod{4}$, $r \ge 5$, let $r = 4s + 1, s \ge 1$.  Consider

\begin{table}[H]
\footnotesize
$PR=$
      \centering
\setlength{\extrarowheight}{1pt}
\setlength{\tabcolsep}{2.5pt}
\begin{tabular}{|c|c|c|c|c|c"c"c|c|c|c|c|c|}
\hline
$1$  & $3$ & $\cdots$  & $4s-5$ &  $4s-3$ & $4s-1$ & $10s+3$ &  $2$ & $4$ &  $\cdots$ & $4s-4$ & $4s-2$ & $4s$ \\
\hline
\end{tabular}
     \centering
\end{table}
\vskip-0.5cm
 \begin{table}[H]
\footnotesize
$QP+QR=$
      \centering
\setlength{\extrarowheight}{1pt}
\setlength{\tabcolsep}{0.8pt}
\begin{tabular}{|c"c|c|c|c|c|c|c"c"c|c|c|c|c|c|}
\hline
$12s+5$ & $12s+3$ & $6s+1$ & $\cdots$ &  $4s+5$  & $10s+5$ & $4s+3$ & $4s+1$ &  $10s+2$ & $8s+1$ &  $\cdots$ &  $6s+5$ & $8s+4$ & $6s+3$ \\
\hline
$12s+4$ & $6s+2$ & $12s+2$ & $\cdots$ & $10s+6$ & $4s+4$ & $10s+4$ & $4s+2$ & $8s+2$ & $10s+1$  & $\cdots$ & $8s+5$ & $6s+4$ & $8s+3$ \\
\hline
\end{tabular}
\end{table}
\vskip-0.3cm


 Clearly, we get a Cartesian tri-magic design with $c_1=8s^2+36s+12$, $c_2=32s^2+27s+6$, and $c_3=18s+6$.\\ \\ 


Finally for $r \equiv 3 \pmod{4}$, $r \ge 7$, let $r = 4s + 3, s \ge 1$.  Consider
\begin{table}[H]
\footnotesize
$PR=$
      \centering
\setlength{\extrarowheight}{1pt}
\setlength{\tabcolsep}{2.5pt}
\begin{tabular}{|c|c|c|c|c|c|c"c|c|c|c|c|c|}
\hline
1& 3& 5& $\cdots$ & \mathversion{bold}${4s-1}$ & $4s+1$ & $4s+3$ &  2 & 4 & $\cdots$ & \mathversion{bold} $4s-2$ & $4s$ & \mathversion{bold} $4s+2$\\
\hline
\end{tabular}
     \centering
\end{table}
\vskip-0.5cm

 \begin{table}[H]
\footnotesize
$QP+QR=$
      \centering
\setlength{\extrarowheight}{1pt}
\setlength{\tabcolsep}{0.8pt}
\begin{tabular}{|c"c|c|c|c|c|c|c"c|c|c|c|c|c|}
\hline
$12s+10$ & $6s+5$ & $12s+8$ & $6s+3$ & $\cdots$ & $10s+10$ & $4s+5$ & $10s+8$ & $8s+6$ & $10s+6$ &  $\cdots$ & \mathversion{bold} $6s+8$ & $8s+8$ & $6s+6$ \\
\hline
$12s+11$ & $12s+9$ & $6s+4$ & $12s+7$ & $\cdots$ & \mathversion{bold} $4s+6$ & $10s+9$ & $4s+4$ &  $10s+7$ & $8s+5$  & $\cdots$ & $8s+9$ & $6s+7$ & \mathversion{bold}$8s+7$ \\
\hline
\end{tabular}
     \centering
\end{table}

We now get $c_1 = 8s^2+38s+27$ and $c_3=18s+15$. However, we have $y_1=32s^2+61s+29$ and $y_2=32s^2+63s+31$.  To make $y_1 = y_2$, we perform the following exchanges.  Note that none of these exchanges would modify the values of $z_k, 1\le k\le r$. Only the value of $c_1$ would be changed.

\begin{enumerate}[(a)]
  \item Interchange the labels $4s-2$ and $6s+8$.  The value of $y_1$ is decreased by $2s+10$.
  \item Interchange the labels $4s-1$ and $4s+6$.  The value of $y_2$ is decreased by 7.
  \item Interchange the labels $4s+2$ and $8s+7$.  The value of $y_2$ is decreased by $4s+5$.
\end{enumerate}

 In total, the value of $y_1$ is decreased by $2s+10$, and the value of $y_2$ is decreased by $4s+12$.  Thus, we now have $c_2=32s^2+59s+19$ and $c_1=8s^2 + 44s + 49$. Clearly, we now have a Cartesian tri-magic design.\\

We now assume $s\ge 2$. Let $A$ be a $2 s\times 2$ magic rectangle using integers in $[0, 4s-1]$. The construction of $A$ can be found in \cite{Reyes}. Hence $r_i(A)=4s-1$ and $c_j(A)=s(4s-1)$. Exchanging columns and rows if necessary, we may assume the $(1,1)$-entry of $A$ is 0, hence the $(1,2)$-entry of $A$ is $4s-1$. \\

Let $\Omega=J_{s,2}\otimes \begin{pmatrix}\alpha\\ \beta\end{pmatrix}$ and $\Theta=A\otimes kJ_{1,k}$, where $\otimes$ denotes the Kronecker's multiplication. Thus $r_i(\Omega)=k(k+1)$, $c_j(\Omega)=s(k+1)$, $r_i(\Theta)=(4s-1)k^2$, and $c_j(\Theta)=s(4s-1)k$ for $1\le i\le 2s$, $1\le j\le 2k$. \\

Let $N=\Omega+\Theta$. Then $r_i(N)=4sk^2+k$ and $c_j(N)=4s^2k+s$ for $1\le i\le 2s$, $1\le j\le 2k$. Now the set of entries of $N$ is $[1,4sk]$. We set $N=(QP|QR)$. Now, the set of remaining integers is $[4sk+1, 4sk+2k-1]$, which will be arranged to form the matrix $PR$. Let
\[\gamma=
\begin{cases}
(\begin{array}{*{6}c|*{7}c} * & 2 & 4 & \cdots & 2k-4 & 2k-2 & 1 & 3 & \cdots & k & \cdots & 2k-3 & 2k-1\end{array}),
& \mbox{if $k$ is odd};\\
(\begin{array}{*{8}c|*{5}c} * & 2 & 4 & \cdots & k& \cdots & 2k-4 & 2k-2 & 1 & 3 & \cdots & 2k-3 & 2k-1\end{array}),
& \mbox{if $k$ is even}.\end{cases}\]

Insert $\gamma+4skJ_{1,2k}$ to the first row of $N$, with $*$ still denoting `$*+4sk$'. The resulting matrix is denoted by $N'$. Now
\begin{equation*}
c_j(N')=\begin{cases}c_j(N)+4sk+(2j-2), & \mbox{if $2\le j\le k$};\\
 c_j(N)+4sk+(2j-1-2k), & \mbox{if $k+1 \le j\le 2k$}.\end{cases}
\end{equation*}
Look at the first row of $N$ which is
\[N^{(1)}=(\begin{array}{*{5}c|*{5}c} 1 & 2 & \cdots & k-1 & k & 4sk-k+1 & 4sk-k+2 & \cdots & 4sk-1 & 4sk\end{array})\]
We swap the $j$-th entry with the $(k+2-j)$-th entry of $N^{(1)}$, for $2\le j\le \lceil k/2 \rceil$ and swap the $j$-th entry with the $(3k+1-j)$-th entry of $N^{(1)}$, for $k+1\le j\le k+\lfloor k/2\rfloor$ to get a new row. It is equivalent to reversing the order of the entries from the 2nd to the $k$-th and reversing the order of the entries from the $(k+1)$-st to the $2k$-th of $N^{(1)}$.
Replace $N^{(1)}$ (i.e., the second row of $N'$) by this new row to get a matrix $M$. Hence $c_j(M)=c_j(N)+4sk+k=4s^2k+4sk+s+k=c_2$, $2\le j\le 2k$. Note that $r_i(M)=r_{i-1}(N)=4sk^2+k=c_3$ for $2\le i\le 2s+1$; $r_1(M)+c_1(M)=r_1(N')+c_1(N)=8sk^2+4s^2k-4sk+2k^2+s-k=c_1$. Clearly $c_1>c_2$ and $c_1>c_3$. Now, $c_3-c_2=s[4k(k-s-1)-1]\ne 0$.
So $M$ corresponds to a tri-magic $(1, 2s, 2k-1)$-board. So we have

\begin{theorem}\label{thm-1eqor} Suppose $q\ge 2$ is even and $r$ is odd. The $(1,q,r)$-board is Cartesian tri-magic.
\end{theorem}

\begin{example}{\rm Consider the graph $K(1, 6, 9)$, i.e., $s=3$ and $k=5$.
Now \[A=\begin{pmatrix}
0 & 11\\
2 & 9\\
6 & 5\\
7 & 4\\
8 & 3\\
10 & 1
\end{pmatrix} \quad
\Omega=\left(\begin{array}{*{5}c|*{5}c}
1 & 2 & 3 & 4 & 5 & 1 & 2 & 3 & 4 & 5\\
5 & 4 & 3 & 2 & 1 & 5 & 4 & 3 & 2 & 1\\
1 & 2 & 3 & 4 & 5 & 1 & 2 & 3 & 4 & 5\\
5 & 4 & 3 & 2 & 1 & 5 & 4 & 3 & 2 & 1\\
1 & 2 & 3 & 4 & 5 & 1 & 2 & 3 & 4 & 5\\
5 & 4 & 3 & 2 & 1 & 5 & 4 & 3 & 2 & 1
\end{array}\right).\]
Hence
\[\Theta=A\otimes 5J_{1,5}=\left(\begin{array}{*{5}c|*{5}c}
0 & 0 & 0 & 0 & 0 & 55 & 55 & 55 & 55 & 55\\
10 & 10 & 10 & 10 & 10 & 45 & 45 & 45 & 45 & 45\\
30 & 30 & 30 & 30 & 30 & 25 & 25 & 25 & 25 & 25\\
35 & 35 & 35 & 35 & 35 & 20 & 20 & 20 & 20 & 20\\
40 & 40 & 40 & 40 & 40 & 15 & 15 & 15 & 15 & 15\\
50 & 50 & 50 & 50 & 50 & 5 & 5 & 5 & 5 & 5
\end{array}\right)\]
and
\[N'=\left(\begin{array}{*{5}c|*{5}c}
* & 62 & 64 & 66 & 68 & 61 & 63 & 65 & 67 & 69\\\hline
1 & 2 & 3 & 4 & 5 & 56 & 57 & 58 & 59 & 60\\
15 & 14 & 13 & 12 & 11 & 50 & 49 & 48 & 47 & 46\\
31 & 32 & 33 & 34 & 35 & 26 & 27 & 28 & 29 & 30\\
40 & 39 & 38 & 37 & 36 & 25 & 24 & 23 & 22 & 21\\
41 & 42 & 43 & 44 & 45 & 16 & 17 & 18 & 19 & 20\\
55 & 54 & 53 & 52 & 51 & 10 & 9 & 8 & 7 & 6
\end{array}\right).\]
Now
\[M=\left(\begin{array}{c|*{4}c*{5}c}
* & 62 & 64 & 66 & 68 & 61 & 63 & 65 & 67 & 69\\\hline
1 & 5 & 4 & 3 & 2 & 60 & 59 & 58 & 57 & 56\\
15 & 14 & 13 & 12 & 11 & 50 & 49 & 48 & 47 & 46\\
31 & 32 & 33 & 34 & 35 & 26 & 27 & 28 & 29 & 30\\
40 & 39 & 38 & 37 & 36 & 25 & 24 & 23 & 22 & 21\\
41 & 42 & 43 & 44 & 45 & 16 & 17 & 18 & 19 & 20\\
55 & 54 & 53 & 52 & 51 & 10 & 9 & 8 & 7 & 6
\end{array}\right).\]
$c_1=768$, $c_2=248$ and $c_3=305$.
}
\end{example}

By a similar way we have
\begin{example}{\rm
The following is a required matrix for $K(1,4,7)$:
\[
\left(\begin{array}{c|ccccccc}
* & 34 & 36 & 38 & 33 & 35 & 37 & 39\\\hline
1 & 4 & 3 & 2 & 32 & 31 & 30 & 29\\
16 & 15 & 14 & 13 & 20 & 19 & 18 & 17\\
21 & 22 & 23 & 24 & 9 & 10 & 11 & 12\\
28 & 27 & 26 & 25 & 8 & 7 &6 & 5\end{array}\right).
\]
$c_1=318$, $c_2=102$ and $c_3=132$.
}\end{example}

 In~\cite{Arumugam}, the authors introduced the concept of local antimagic chromatic number of a graph $G$, denoted $\chi_{la}(G)$. Observe that for every complete tripartite graph $K(p,q,r)$, $\chi_{la}(K(p,q,r))=3$ if and only if the $(p,q,r)$-board is Cartesian tri-magic. Thus, we have obtained various sufficient conditions such that $\chi_{la}(K(p,q,r)) = 3$. Interested readers may refer to~\cite{Bensmail, LNS, LSN1, LSN2} for more results on local antimatic chromatic number of graphs. 

\section{Cartesian bi-magic}\indent

\begin{theorem}\label{thm-11r} The $(1,1,r)$-board is Cartesian bi-magic if and only if $r\not\equiv1\pmod{4}$. \end{theorem}

\begin{proof}
{[Sufficiency]} Suppose $r\not\equiv 1\pmod{4}$. We have three cases.
\begin{enumerate}[1.]
\item
Suppose $r\equiv0\pmod{4}$. Assign $2r+1$ to $PQ$. The assignments to $PR$ and $QR$ are given by row 1 and row 2 respectively in the matrix below.

\vskip-0.3cm
 \begin{table}[H]
\footnotesize
\setlength{\extrarowheight}{1pt}
\setlength{\tabcolsep}{2pt}
\quad \begin{tabular}{r|c|c|c|c"c|c|c|c"c"c|c|c|c"c|c|c|c|}
\cline{2-18}
$PR=$ & 1 & $2r-1$ & $2r-2$ & 4 & 5 & $2r-5$ & $2r-6$ & 8 & $\cdots$ & $r-7$ & $r+7$ & $r+6$ & $r-4$ &  $r-3$ & $r+3$ & $r+2$ & $r$ \\
\cline{2-18}
$QR=$ & $2r$ & 2 & 3 & $2r-3$ & $2r-4$ & 6 & 7 & $2r-7$ & $\cdots$ & $r+8$ & $r-6$ &  $r-5$ & $r+5$  & $r+4$ & $r-2$ & $r-1$ & $r+1$ \\
\cline{2-18}
\end{tabular}
\end{table}

\vskip-0.6cm
Clearly, $c_1=c_2= r^2 + 5r/2 + 1$ and $c_3=2r+1$.
\item Suppose $r\equiv2\pmod{4}$. For $r=2$, assign $1$ to $PQ$, assign $2$ and $4$ to the only row of $PR$ and assign $5$ and $3$ to the only row of $QR$. Clearly, $c_1=c_3=7$, $c_2=9$. For $r\ge 6$, assign $1$ to $PQ$. The assignments to $PR$ and $QR$ are given by row 1 and row 2 respectively in the matrix below. Note that if $r\ge 10$, we would assign from the 7th column to the last column in a way similar to that for $r\equiv0\pmod{4}$. 
    \vskip-0.3cm
\begin{table}[H]
\footnotesize
\setlength{\extrarowheight}{1pt}
\setlength{\tabcolsep}{2pt}
\quad \begin{tabular}{r|c|c|c|c|c|c||c|c|c|c"c"c|c|c|c|}
\cline{2-16}
$PR=$ & $r-4$ & $r+6$ &  $r-2$ & $r+4$ & $r+3$ & $r+2$ & 2 & $2r$ & $2r-1$ & 5 & $\cdots$ & $r-8$ & $r+10$ & $r+9$ & $r-5$ \\
\cline{2-16}
$QR=$ &  $r+7$ & $r-3$  & $r+5$ & $r-1$ & $r$ & $r+1$ & $2r+1$ & 3 & 4 & $2r-2$ & $\cdots$ & $r+11$ & $r-7$ & $r-6$ & $r+8$ \\
\cline{2-16}
\end{tabular}
\end{table}
\vskip-0.6cm

Clearly, $c_1=c_2= r^2+3r/2+1$ and $c_3=2r+3$.

\item Suppose $r\equiv3\pmod{4}$. Assign $r+1$ to $PQ$. The assignments to $PR$ and $QR$ are given by row 1 and row 2 respectively in the matrix below. Note that if $r\ge 7$, we would assign from the 4th column to the last column in a way similar to that for $r\equiv0\pmod{4}$.
\vskip-0.3cm
 \begin{table}[H]
\footnotesize
\setlength{\extrarowheight}{1pt}
\setlength{\tabcolsep}{2pt}
\quad \begin{tabular}{r|c|c|c||c|c|c|c"c"c|c|c|c|}
\cline{2-13}
$PR=$ & $r-2$ & $r+3$ &  $r+2$ & 1 & $2r$ & $2r-1$ & 4 & $\cdots$ & $r-6$ & $r+7$ & $r+6$ & $r-3$\\
\cline{2-13}
$QR=$ & $r+4$ & $r-1$  & $r$ & $2r+1$ & 2 & 3 & $2r-2$ & $\cdots$ & $r+8$ & $r-5$ & $r-4$ & $r+5$ \\
\cline{2-13}
\end{tabular}
\end{table}
\vskip-0.6cm

Clearly, $c_1=c_2= r^2+2r+1$ and $c_3=2r+2$.
\end{enumerate}

[Necessity]
Suppose there is a Cartesian bi-magic $(1,1,r)$-board. Clearly, $r>1$.\\
Since $rc_3=(r+1)(2r+1)-s(PQ)$, $2r+1\le c_3\le 2r+3$. So we have three cases.
\begin{enumerate}[(1)]
\item Suppose $c_3=2r+1$. In this case, $s(PQ)=2r+1$. It follows that $c_1=c_2\ne c_3$. Hence, we must divide $[1,2r]$ into two disjoint sets of $r$ integers with equal total sums. Hence, $r(2r+1)/2$ is even. So that $r$ must be even.

\item Suppose $c_3=2r+2$. In this case, $s(PQ)=r+1$. Hence, we must divide $[1, r] \cup [r + 2, 2r + 1]$ into two disjoint sets of $r$ integers, say $A_1$ and $A_2$, such that (a) $s(A_1) = s(A_2) \not= r + 1$, or (b) $s(A_1) = r + 1$, $s(A_2) \not= r + 1$, where $s(A_i)$ denotes the sum of all integers in $A_i$, $i=1,2$. For both case, $s(A_1)+s(A_2)=2r(r+1)$.
\begin{enumerate}[(a)]
    \item Suppose $s(A_1) = s(A_2) \ne r + 1$. This implies that $s(A_1)=r(r+1)$. Since $c_3=2r+2$, integers in $[1, r] \cup [r + 2, 2r + 1]$ must be paired as $(1, 2r + 1), (2, 2r), \ldots, (r, r + 2)$ as the corresponding entries in the two $1 \times r$ rectangles $PR$ and $QR$.
    Let $PR=\begin{pmatrix} a_1 & a_2 & \cdots & a_r\end{pmatrix}$ and $QR=\begin{pmatrix} b_1 & b_2 & \cdots & b_r\end{pmatrix}$. Here $a_i+b_i=c_3=2r+2$, $1\le i\le r$. Also we have $\sum\limits_{i=1}^r |a_i-b_i|=\sum\limits_{i=1}^r 2i=r(r+1)$.
    Without loss of generality, we may assume
    $a_i>b_i$ when $1\le i\le k$ and $b_j>a_j$ when $k+1\le j\le r$, for some $k$. Thus,
    \begin{equation}\label{eq-case2}r(r+1)=\sum_{i=1}^r |a_i-b_i|=\sum_{i=1}^k (a_i-b_i)+\sum_{j=k+1}^r (b_j-a_j).\end{equation}

    Now $0=s(A_1)-s(A_2)=\sum\limits_{i=1}^k (a_i-b_i)+\sum\limits_{j=k+1}^r (a_j-b_j)=\sum\limits_{i=1}^k (a_i-b_i)-\sum\limits_{j=k+1}^r (b_j-a_j)$. Hence we get $\sum\limits_{i=1}^k (a_i-b_i)=\sum\limits_{j=k+1}^r (b_j-a_j)$. Combining with \eqref{eq-case2} we have $r(r+1)=2\sum\limits_{i=1}^k (a_i-b_i)$. Since each $a_i-b_i$ is even, $r(r+1)\equiv 0\pmod{4}$. Hence $r\equiv 0,3\pmod{4}$.

\item The sum of $r$ distinct positive integers is at least $1 + \cdots + r = r(r+1)/2$. So $s(A_1)=r+1\ge r(r+1)/2$. Hence $r=2$.
\end{enumerate}

\item Suppose $c_3=2r+3$. In this case, $s(PQ)=1$. Hence, we must divide $[2,2r+1]$ into two disjoint sets of $r$ integers, say $A_1$ and $A_2$, such that (a) $s(A_1) =s(A_2) \not= 2r+2$, or (b) $s(A_1)=2r+2$ and $s(A_2)\not=2r+2$. For these case, $s(A_1)+s(A_2)=r(2r+3)$.

    In (a), we get $2s(A_1)=r(2r+3)$. Hence $r$ is even.

In (b), similar to the case~(2)(b) above, $2r+2=s(A_1)\ge (r+3)r/2$. Thus $r(r-1)\le 4$, and hence  $r=2$.
 \end{enumerate}
    \end{proof}

\begin{theorem}\label{thm-22rbimagic} The $(2,2,r)$-board is Cartesian bi-magic for even $r$.  \end{theorem}

\begin{proof} For $r=2$, a required design with $c_1=c_3=24$, and $c_2=30$ is given below.
{\footnotesize 
\begin{equation*}PQ=\begin{array}{|c|c|}
\hline
4 & 10 \\
\hline
5 & 11  \\
\hline
\end{array}\qquad
PR=\begin{array}{|c|c|}
\hline
8 & 2  \\
\hline
1  & 7 \\
\hline
\end{array}\qquad
QR=\begin{array}{|c|c|}
\hline
12  & 9   \\
\hline
3 & 6 \\
\hline
\end{array}
\end{equation*}
}

For $r\ge 4$, let
{\footnotesize
\begin{equation*}PQ=\begin{array}{|c|c|}
\hline
4r+1 & 4r+2 \\
\hline
4r+4 & 4r+3 \\
\hline
\end{array}\end{equation*}}

For $r\equiv0\pmod{4}$, the assignments to $PR$ and $QR$ are then given by rows $1,2$ and rows $3,4$ respectively below.
\vskip1mm

{\footnotesize
\setlength{\extrarowheight}{1pt}
\setlength{\tabcolsep}{2.4pt}
\centerline{\begin{tabular}{c|c|c"c|c"c"c|c"c"c|c"c|c|}
\cline{2-13}
\multirow{2}{1cm}{$PR=$} & 1 & 8 & 9 & 16 & $\cdots$ & \mathversion{bold} $2r$ & $2r+1$ & $\cdots$ & $4r-15$ & $4r-8$ & $4r-7$ & $4r$ \\
\cline{2-13}
& $4r-1$ & $4r-6$ & $4r-9$ & $4r-14$ & $\cdots$ & \mathversion{bold} $2r+2$ & $2r-1$ & $\cdots$ & 15 & 10 & 7 & 2\\
\cline{2-13}
\multirow{2}{1cm}{$QR=$} & $4r-2$ & $4r-5$ & $4r-10$ & $4r-13$ & $\cdots$ & $2r+3$ & $2r-2$ & $\cdots$ & 14 & 11 & 6 & 3 \\
\cline{2-13}
& 4 & 5 & 12 & 13 & $\cdots$ & $2r-3$ & $2r+4$ & $\cdots$ & $4r-12$ & $4r-11$ & $4r-4$ & $4r-3$ \\
\cline{2-13}
\end{tabular}}}

\vskip1mm
Now, interchange entries $2r$ and $2r+2$. We get a Cartesian bi-magic design with $c_1 = c_2 = 2r^2 + 17r/2 + 5$, and $c_3 = 8r+2$.\\

For $r \equiv 2 \pmod{4}$, the assignments to $PR$ and $QR$ are then given by rows $1,2$ and rows $3,4$ respectively below.
\vskip1mm
\centerline{
{\footnotesize
\setlength{\extrarowheight}{1pt}
\setlength{\tabcolsep}{2.4pt}
\begin{tabular}{c|c|c"c|c"c"c|c"c"c|c"c|c|}
\cline{2-13}
\multirow{2}{1cm}{$PR=$} & 4 & 5 & 12 & 13 & $\cdots$ & \mathversion{bold} $2r$ & $2r+1$ & $\cdots$ & $4r-12$ & $4r-11$ & $4r-4$ & $4r-3$ \\
\cline{2-13}
& $4r-2$ & $4r-5$ & $4r-10$ & $4r-13$ & $\cdots$ & \mathversion{bold} $2r+2$ & $2r-1$ & $\cdots$ & 14 & 11 & 6 & 3 \\
\cline{2-13}
\multirow{2}{1cm}{$QR=$} & $4r-1$ & $4r-6$ & $4r-9$ & $4r-14$ & $\cdots$ & $2r+3$ & $2r-2$ & $\cdots$ & 15 & 10 & 7 & 2\\
\cline{2-13}
& 1 & 8 & 9 & 16 & $\cdots$ & $2r-3$ & $2r+4$ & $\cdots$ & $4r-15$ & $4r-8$ & $4r-7$ & $4r$ \\
\cline{2-13}
\end{tabular}
}}
\vskip1mm

Now, interchange entries $2r$ and $2r+2$. We also get a Cartesian bi-magic design with $c_1 = c_2 = 2r^2 + 17r/2 + 5$, and $c_3 = 8r+2$.   \end{proof}

\begin{theorem}\label{thm-evenpppbimagic}  The $(p,p,p)$-board is Cartesian bi-magic for all even $p\ge 2$. \end{theorem}

\begin{proof} For $p=2$, a required design is given in the proof of Theorem~\ref{thm-22rbimagic}. For $p=4$, a required design is given as follows with $c_1=c_3=180$, and $c_2=228$.
{\footnotesize
\setlength{\tabcolsep}{4pt}
\begin{equation*}
PQ=\begin{tabular}{|c|c|c|c|}
\hline
41 & 14  &   15   &  44   \\
\hline
20 & 39   &  38   &  17 \\
\hline
40   &  19  &  18  &  37  \\
\hline
13   &   42   &   43    &  16\\
\hline
\end{tabular}\qquad
PR=
\begin{tabular}{|c|c|c|c|}
\hline
29 & 2 &  3 &  32  \\
\hline
8  & 27  & 26  & 5 \\
\hline
28  &   7  &   6  &  25\\
\hline
1  &  30  & 31  & 4\\
\hline
\end{tabular}\qquad
QR=\begin{tabular}{|c|c|c|c|}
\hline
45 & 10   &  11  &   48 \\
\hline
24  & 35  &  34  &  21   \\
\hline
36  &  23   &  22   & 33\\
\hline
9   &   46  &  47  &  12\\
\hline
\end{tabular}
\end{equation*}}

For even $p = 2n \ge 6$, we can get a $(p,p,p)$-design that is Cartesian magic as follows.
\begin{enumerate}[(1).]
  \item Begin with a $PQ$, $PR$ and $QR$ each of size $2\times 2$.
  \item Substitute each entry of the above matrices by a magic square of order $n$ using the integers in the given interval accordingly assigned below.
\end{enumerate}
 {\footnotesize
\setlength{\tabcolsep}{2pt}
\begin{equation*}
PQ  =
\begin{tabular}{|c|c|}
\hline
(4) [$3n^2+1$,$4n^2$] & (10) [$9n^2+1$,$10n^2$] \\
\hline
(5) [$4n^2+1$,$5n^2$]  & (11) [$10n^2+1$,$11n^2$]  \\
\hline
\end{tabular}
\qquad
PR  =
\begin{tabular}{|c|c|}
\hline
(8) [$7n^2+1$,$8n^2$]  & (2) [$n^2+1$,$2n^2$]  \\
\hline
(1) [1,$n^2$]  & (7) [$6n^2+1$,$7n^2$] \\
\hline
\end{tabular}
\end{equation*}
\begin{equation*}
QR =
\begin{tabular}{|c|c|}
\hline
(12) [$11n^2+1$,$12n^2$] & (9) [$8n^2+1$,$9n^2$]  \\
\hline
(3) [$2n^2+1$,$3n^2$]  & (6) [$5n^2+1$,$6n^2$]  \\
\hline
\end{tabular}
\end{equation*}}

Thus, we can get a required Cartesian bi-magic design with $c_1=c_3=22n^3+2n$ and $c_2=28n^3+2n$.   \end{proof}

\begin{theorem}\label{thm-oddpppbimagic}  The $(p,p,p)$-board is Cartesian bi-magic for all odd $p\ge 3$. \end{theorem}

\begin{proof} Let $p = 2n + 1$ be odd.  Let $M$ be a $p \times p$ magic square, with each row sum is equal to each column sum which is $p(p^2+1)/2$.

\noindent Define matrices $A$, $B$, and $C$, each a $p \times p$ matrix, as follows.  The entries in $A$ are filled row by row.  In the first row of $A$, $A_{1j} = 2$ for $j\in [1, n]$, $A_{1,n+1} = 1$, and $A_{1j} = 0$ for $j \in [n + 2, 2n + 1]$.  Beginning with the second row of $A$, $A_{i,j} = A_{i-1,j-1}$, and $A_{i1} = A_{i-1,p}$ for $i, j\in [2, p]$.  The matrix $B$ is formed column by column such that  $B_{i,j} = A_{i,n+1+j}$, for $j \in [1, n]$, and $B_{i,j} = A_{i,j-n}$, for $j \in [n + 1, 2n + 1]$.  Clearly, beginning with the second row of $B$, each row can be obtained from the previous row using the rearrangements as in the rows of $A$.  The entries in $C$ are filled row by row.  In the first row of $C$, $C_{1j} = 1$ for $j \in [1, n]$, $C_{1,n+1} = 0$, $C_{1j} = 1$ for $j\in [n + 2, 2n]$, and $C_{1,2n+1} = 2$.  Beginning with the second row of $C$, each row is obtained from the previous row using the rearrangements as in the rows of $A$.

\noindent Observe that, in each of the matrices $A$, $B$, and $C$, each row sum and each column sum is equal to $p$.  In addition, for $i, j \in [ 1,p]$, $\{A_{i,j}, B_{i,j}, C_{i,j}\} = \{0, 1, 2\}$.

\noindent We now define the matrices $PQ$, $PR$, and $QR$ as follows.   For $i, j \in [1, p]$,
$$(PQ)_{i,j} =\begin{cases}
3M_{i,j} - A_{i,j}+1 & \mbox{ if } i=j, \\
3M_{i,j} - A_{i,j} &\mbox{ if } i\not= j,
\end{cases}$$
$$(PR)_{i,j} =\begin{cases}
3M_{i,j} - B_{i,j}-2 & \mbox{ if } i=j, \\
3M_{i,j} - B_{i,j} & \mbox{ if } i\not=j,
\end{cases}$$ and
$$(QR)_{i,j} = \begin{cases}
3M_{i,j} - C_{i,j}+1 & \mbox{ if } i=j,\\
3M_{i,j} - C_{i,j} & \mbox{ if } i\not=j.
\end{cases}$$ In addition, note that in $PQ$ and $QR$, each row sum is equal to each column sum and is equal to $3p(p^2+1)/2 - p + 1$ while in $PR$, each row sum is equal to each column sum and is equal to $3p(p^2+1)/2 - p - 2$. Hence, the $(p,p,p)$-board is Cartesian bi-magic with $c_1=c_3=3p(p^2+1)-2p-1$ and $c_2=3p(p^2+1)-2p+2$.\end{proof}

\noindent We now provide the example of $p=7$.
{\footnotesize \setlength{\tabcolsep}{4pt}
\begin{equation*}M=
\begin{tabular}{|c|c|c|c|c|c|c|}
\hline
30 & 39 & 48 & 1 & 10 & 19 & 28 \\
\hline
38  & 47  & 7  & 9  & 18  & 27  & 29 \\
\hline
46 & 6  &  8 & 17  & 26  & 35  & 37 \\
\hline
5  & 14  &  16 & 25  & 34  & 36  & 35  \\
\hline
13  & 15  &  24 & 33  & 42  & 44  &  4 \\
\hline
21  & 23  &  32 & 41  & 43  & 3  &  12 \\
\hline
22  & 31  &  40 & 49  & 2  & 11  &  20\\
\hline
\end{tabular}\end{equation*}
\begin{align*}
A=
\begin{tabular}{|c|c|c|c|c|c|c|}
\hline
2 & 2 & 2 & 1 & 0 & 0 & 0 \\
\hline
0 & 2 & 2 & 2 & 1 & 0 & 0 \\
\hline
0 & 0 & 2 & 2 & 2 & 1 & 0 \\
\hline
0 & 0 & 0 & 2 & 2 & 2 & 1 \\
\hline
1 & 0 & 0 & 0 & 2 & 2 & 2 \\
\hline
2 & 1 & 0 & 0 & 0 & 2 & 2 \\
\hline
2 & 2 & 1 & 0 & 0 & 0 & 2 \\
\hline
\end{tabular}
\qquad &
PQ=
\begin{tabular}{|c|c|c|c|c|c|c|}
\hline
89 & 115 & 142 & 2 & 30 & 57 & 84 \\
\hline
114 & 140 & 19 & 25 & 53 & 81 & 87 \\
\hline
138 & 18 & 23 & 49 & 76 & 104 & 111 \\
\hline
15 & 42 & 48 & 74 & 100 & 106 & 134 \\
\hline
38 & 45 & 72 & 99 & 125 & 130 & 10 \\
\hline
61 & 68 & 96 & 123 & 129 & 8 & 34 \\
\hline
64 & 91 & 119 & 147 & 6 & 33 & 59 \\
\hline
\end{tabular}\\
B=
\begin{tabular}{|c|c|c|c|c|c|c|}
\hline
0 & 0 & 0 & 2 & 2 & 2 & 1  \\
\hline
1 & 0 & 0 & 0 & 2 & 2 & 2  \\
\hline
2 & 1 & 0 & 0 & 0 & 2 & 2  \\
\hline
2 & 2 & 1 & 0 & 0 & 0 & 2  \\
\hline
2 & 2 & 2 & 1 & 0 & 0 & 0 \\
\hline
0 & 2 & 2 & 2 & 1 & 0 & 0 \\
\hline
0 & 0 & 2 & 2 & 2 & 1 & 0 \\
\hline
\end{tabular}
\qquad &
PR=
\begin{tabular}{|c|c|c|c|c|c|c|}
\hline
88 & 117 & 144 & 1 & 28 & 55 & 83 \\
\hline
113 & 139 & 21 & 27 & 52 & 79 & 85 \\
\hline
136 & 17 & 22 & 51 & 78 & 103 & 109 \\
\hline
13 & 40 & 47 & 73 & 102 & 108 & 133 \\
\hline
37 & 43 & 70 & 98 & 124 & 132 & 12 \\
\hline
63 & 67 & 94 & 121 & 128 & 7 & 36 \\
\hline
66 & 93 & 118 & 145 & 4 & 32 & 58 \\
\hline
\end{tabular}\\
C=
\begin{tabular}{|c|c|c|c|c|c|c|}
\hline
1 & 1 & 1 & 0 & 1 & 1 & 2 \\
\hline
2 & 1 & 1 & 1 & 0 & 1 & 1 \\
\hline
1 & 2 & 1 & 1 & 1 & 0 & 1 \\
\hline
1 & 1 & 2 & 1 & 1 & 1 & 0 \\
\hline
0 & 1 & 1 & 2 & 1 & 1 & 1 \\
\hline
1 & 0 & 1 & 1 & 2 & 1 & 1 \\
\hline
1 & 1 & 0 & 1 & 1 & 2 & 1 \\
\hline
\end{tabular}
\qquad & QR=
\begin{tabular}{|c|c|c|c|c|c|c|}
\hline
90 & 116 & 143 & 3 & 29 & 56 & 82 \\
\hline
112 & 141 & 20 & 26 & 54 & 80 & 86 \\
\hline
137 & 16 & 24 & 50 & 77 & 105 & 110 \\
\hline
14 & 41 & 46 & 75 & 101 & 107 & 135 \\
\hline
39 & 44 & 71 & 97 & 126 & 131 & 11 \\
\hline
62 & 69 & 95 & 122 & 127 & 9 & 35 \\
\hline
65 & 92 & 120 & 146 & 5 & 31 & 60 \\
\hline
\end{tabular}
\end{align*}
}

\begin{theorem} For $p \ge 3$, (i) the $(p,p,r)$-board is Cartesian bi-magic when $r = p$ or $r$ is even, (ii) the $(p,r,r)$-board is Cartesian bi-magic for even $p$.  \end{theorem}

\begin{proof} (i) View the board as containing a $(2p,r)$-rectangle and another $(p,p)$-square. Since $p^2 < 2pr$, we first assign integers in $[1, p^2]$ to the $(p,p)$-rectangle to get a $p\times p$ magic square with magic constant $p(p^2+1)/2$. Next we assign integers in $[p^2+1, p^2 + 2pr]$ to get a magic $(2p,r)$-rectangle with row sum constant $r(2pr+1)/2 + p^2r$ and column sum constant $p(2pr+1)+2p^3$. Note that the existence of magic rectangle of even order can be found in \cite{Reyes}. Hence, the assignment we have now is Cartesian bi-magic with $c_1 = c_2 = p(p^2+1)/2 + r(2p^2 + 2pr + 1)/2 \not= c_3 = 2p^3 + 2p^2r + p$.

(ii) We repeat the approach as in (i). Begin with the $(p,2r)$-rectangle and then the $(r,r)$-rectangle if $2pr < r^2$. Otherwise, reverse the order.    \end{proof}

\section{Cartesian magic}\label{sec-magic}\indent

Clearly, the (1,1,1)-board is not Cartesian magic. In this section, we always assume $(p,q,r)\not=(1,1,1)$. Let $m$ be the magic constant of a Cartesian magic $(p,q,r)$-board. Throughout this paper,  we will use $s(PQ)$, $s(PR)$ and $s(QR)$ to denote the sum of integers in $PQ$, $PR$ and $QR$, respectively.

\begin{lemma}\label{lem-magicN1} If a $(p,q,r)$-board is Cartesian magic, then $p+q+r$ divides $(pq+pr+qr)(pq+pr+qr+1)$.  \end{lemma}

\begin{proof} It follows from the fact that $(p+q+r)m = 2[1+2+\cdots+(pq+pr+qr)]$.  \end{proof}

\begin{lemma}\label{lem-magicN2} If a $(p,q,r)$-board is Cartesian magic, then
\begin{enumerate}[(i)]
\item $s(PQ)+s(PR)=mp$, $s(PQ)+s(QR)= mq$, $s(PR)+s(QR)= mr$;
\item $s(QR)-s(PR)=m(q-p)$;
\item $s(QR)-s(PQ)=m(r-p)$;
\item $s(PR)-s(PQ)=m(r-q)$.
\end{enumerate}
\end{lemma}

\begin{proof} By definition, we get (i). Clearly,
(ii), (iii) and (iv) follow from (i). \end{proof}

\begin{theorem} If a $(p,p,pr)$-board is Cartesian magic for $p, r\ge 1$, then $r=1$. \end{theorem}

\begin{proof} Under the hypothesis, by Lemma~\ref{lem-magicN1}, $m=(2p^2r+p^2)(2p^2r+p^2+1)/(pr+2p)$.
By Lemma~\ref{lem-magicN2}(i), $s(PR)+s(QR)= prm$. Since $s(PQ)\ge 1$, $s(PR)+s(QR)$ is always less than the sum of all labels. That is, $r(2p^2r+p^2)(2p^2r+p^2+1)/(r+2)<(2p^2r+p^2)(2p^2r+p^2+1)/2$. We have $\frac{r}{r+2}<\frac{1}{2}$. Hence $r=1$.  \end{proof}

\begin{theorem}\label{thm-magic1qr} There is no Cartesian magic $(1,q,r)$-board for all $r\ge q\ge 1$. \end{theorem}

\begin{proof} Suppose there is a Cartesian magic $(1,q,r)$-board.
By Lemma~\ref{lem-magicN2}(i), $s(PR)+s(QR)=mr = r(q+1)(r+1)(q+r+qr)/(1+q+r)$. Thus we have
$r(q+1)(r+1)(q+r+qr)/(1+q+r)< (q+1)(r+1)(q+r+qr)/2$. This implies that $r<1+q$, hence $r=q$.
By Lemma~\ref{lem-magicN2}(i) and (iv) we know that $m$ is an even number. So $m/2=\frac{q(q+1)^2(q+2)}{4q+2}=q^2+q+\frac{q^4-q^2}{2(2q+1)}$ is an integer. So $2q+1$ is a factor of $q^2(q+1)(q-1)$.
Since gcd$(2q+1, q)=1$ and gcd$(2q+1, q+1)=1$, $2q+1$ is a factor of $q-1$. It is impossible except when $q=1$.
But when $q=1$, we know that there is no Cartesian magic $(1,1,1)$-board. This completes the proof.
\end{proof}

From \cite{Shiu2002}, we know that $K(p,p,p)\cong C_3\circ N_p$, the lexicographic product of $C_3$ and $N_p$, is supermagic. That is, $(p,p,p)$-board is Cartesian magic for $p\ge 2$. In the following theorems, we provide another Cartesian magic labeling.

\begin{theorem}\label{thm-evenpppmagic} The $(p,p,p)$-board is Cartesian magic for all even $p \ge 2$. \end{theorem}

\begin{proof} For $p=2$, a required design with $m=26$ is given by the rectangles below.
{\setlength{\tabcolsep}{4pt}
\footnotesize
\begin{equation*}PQ=
\begin{tabular}{|c|c|}
\hline
8 & 5 \\
\hline
6 & 7  \\
\hline
\end{tabular}
\qquad
PR=\begin{tabular}{|c|c|}
\hline
4 & 9  \\
\hline
1  & 12 \\
\hline
\end{tabular}\qquad
QR=
\begin{tabular}{|c|c|}
\hline
10  & 2   \\
\hline
11 & 3 \\
\hline
\end{tabular}\end{equation*}
}
For $p=4$, a required design with $m=196$ is given below.
{\setlength{\tabcolsep}{4pt}
\footnotesize
\begin{equation*}PQ=
\begin{tabular}{|c|c|c|c|}
\hline
32 & 18  &   19   &   29   \\
\hline
25 & 23   &  22   &  28\\
\hline
17   &  31  &  30  &  20  \\
\hline
24   &   26   &   27    &   21\\
\hline
\end{tabular}\qquad
PR=
\begin{tabular}{|c|c|c|c|}
\hline
48 & 2 &  3 &  45  \\
\hline
33  & 15  & 14  & 36 \\
\hline
1  &   47  &   46  &  4\\
\hline
16  &  34  & 35  & 13\\
\hline
\end{tabular}
  \qquad
QR=
\begin{tabular}{|c|c|c|c|}
\hline
44 & 6   &  7  &   41 \\
\hline
37  & 11  &  10  &  40   \\
\hline
5  &  43   &   42   & 8\\
\hline
12   &   38  &  39  &  9\\
\hline
\end{tabular}\end{equation*}
}

For $p=2n\ge 6$, using the approach as in the proof of Theorem~\ref{thm-evenpppbimagic} and the (2,2,2)-design as above. Clearly, the Cartesian magic constant thus obtained is $m = p(3p^2 + 1)$.  \end{proof}

For example, a Cartesian magic (6,6,6)-design is given below with $m=654$.
{\footnotesize
\begin{equation*}
PQ=
\begin{tabular}{|c|c|c"c|c|c|}
\hline
65 & 72 & 67 & 38 & 45 & 40 \\
\hline
70 & 68 & 66 & 43 & 41 & 39 \\
\hline
69 & 64 & 71 & 42 & 37 & 44 \\
\thickhline
47 & 54 & 49 & 56 & 63 & 58 \\
\hline
52 & 50 & 48 & 61 & 59 & 47 \\
\hline
51 & 46 & 53 & 60 & 55 & 62 \\
\hline
\end{tabular}\quad
PR=\begin{tabular}{|c|c|c"c|c|c|}
\hline
 29 & 36 & 31 & 74 & 81 & 76 \\
\hline
34 & 32 & 30  & 79 & 77 & 75 \\
\hline
  33 & 28 & 35  & 78 & 73 & 80 \\
\thickhline
    2 & 9 & 4 &  101 & 108 & 103 \\
\hline
     7 & 5 & 3 & 106 & 104 & 102 \\
\hline
     6 & 1 & 8 &  105 & 100 & 107 \\
\hline
\end{tabular}
\end{equation*}
\begin{equation*}
\quad QR=
\begin{tabular}{|c|c|c"c|c|c|}
\hline
83 & 90 & 85 &  11 & 18 & 13  \\
\hline
88 & 86 & 84 &  16 & 14 & 12  \\
\hline
87 & 82 & 89 &  15 & 10 & 17  \\
\thickhline
 92 & 99 & 94 & 20 & 27 & 22\\
\hline
 97 & 95 & 93 & 25 & 23 & 21\\
\hline
 96 & 91 & 98 & 24 & 19 & 26\\
\hline
\end{tabular}
\end{equation*}
}
\begin{theorem}\label{thm-oddpppmagic} The $(p,p,p)$-board is Cartesian magic for all odd $p \ge 3$. \end{theorem}

\begin{proof}  Let $A$, $B$, $C$ and $M$ be as defined in the proof of Theorem~\ref{thm-oddpppbimagic}.
We now define the matrices $PQ$, $PR$, and $QR$ as follows. For $i, j \in[1, p]$, $(PQ)_{i,j} = 3M_{i,j} - A_{i,j}$, $(PR)_{i,j} = 3M_{i,j} - B_{i,j}$, and $(QR)_{i,j} = 3M_{i,j} - C_{i,j}$.  Thus, $\{(PQ)_{i,j}, (PR)_{i,j}, (QR)_{i,j} : i, j \in [1, p]\} = \{3M_{i,j} - k : i, j \in [1, p], k \in [0,2]\} = [1,3p^2]$.   In addition, note that each row sum and each column sum is equal to $3p(p^2+1)/2 - p$. Hence, the $(p,p,p)$-board is Cartesian magic with $m=3p(p^2+1)-2p$.    \end{proof}

\noindent We now provide the example of $p=7$. Only matrices $PQ, PR$ and $QR$ are shown.
{\footnotesize
\setlength{\tabcolsep}{3pt}
\begin{equation*}
PQ=
\begin{tabular}{|c|c|c|c|c|c|c|}
\hline
88 & 115 & 142 & 2 & 30 & 57 & 84 \\
\hline
114 & 139 & 19 & 25 & 53 & 81 & 87 \\
\hline
138 & 18 & 22 & 49 & 76 & 104 & 111 \\
\hline
15 & 42 & 48 & 73 & 100 & 106 & 134 \\
\hline
38 & 45 & 72 & 99 & 124 & 130 & 10 \\
\hline
61 & 68 & 96 & 123 & 129 & 7 & 34 \\
\hline
64 & 91 & 119 & 147 & 6 & 33 & 58 \\
\hline
\end{tabular}\qquad
PR=
\begin{tabular}{|c|c|c|c|c|c|c|}
\hline
90 & 117 & 144 & 1 & 28 & 55 & 83 \\
\hline
113 & 141 & 21 & 27 & 52 & 79 & 85 \\
\hline
136 & 17 & 24 & 51 & 78 & 103 & 109 \\
\hline
13 & 40 & 47 & 75 & 102 & 108 & 133 \\
\hline
37 & 43 & 70 & 98 & 126 & 132 & 12 \\
\hline
63 & 67 & 94 & 121 & 128 & 9 & 36 \\
\hline
66 & 93 & 118 & 145 & 4 & 32 & 60 \\
\hline
\end{tabular}
\end{equation*}
\begin{equation*}
QR=
\begin{tabular}{|c|c|c|c|c|c|c|}
\hline
89 & 116 & 143 & 3 & 29 & 56 & 82 \\
\hline
112 & 140 & 20 & 26 & 54 & 80 & 86 \\
\hline
137 & 16 & 23 & 50 & 77 & 105 & 110 \\
\hline
14 & 41 & 46 & 74 & 101 & 107 & 135 \\
\hline
39 & 44 & 71 & 97 & 125 & 131 & 11 \\
\hline
62 & 69 & 95 & 122 & 127 & 8 & 35 \\
\hline
65 & 92 & 120 & 146 & 5 & 31 & 59 \\
\hline
\end{tabular}
\end{equation*}}

\section{Miscellany and unsolved problems}\indent

Here are some ad hoc examples:

\begin{example}{\rm
Modifying a $3\times 3$ magic square we get a Cartesian bi-magic labeling for the (1,2,3)-board whose corresponding labeling matrix is
\begin{equation*}
\left(\begin{array}{c|cc|ccc}
* & 1 & 2 & 5 & 11 & 4\\\hline
1 & * & * & 6 & 7 & 9\\
2 & * & * & 10 & 3 & 8\\\hline
5 & 6 & 10 & * & * & *\\
11 & 7 & 3 & * & * & *\\
4 & 9 & 8 & * & * & *\\
\end{array}\right)
\left(\begin{array}{c}
23\\\hline
23\\
23\\\hline
21\\21\\21
\end{array}\right)
\end{equation*}
}
\end{example}

\begin{example}{\rm
 A Cartesian tri-magic labeling for the (1,3,3)-board whose corresponding labeling matrix is
\begin{equation*}
\left(\begin{array}{c|ccc|ccc}
* & 10 & 13 & 7 & 12 & 4 & 8\\\hline
10 & * & * & * & 1 & 6 & 15\\
13 & * & * & * & 3 & 11 & 5\\
7 & * & * & * & 14 & 9 & 2\\\hline
12 & 1 & 3 & 14 & * & * & *\\
4 & 6 & 11 & 9 & * & * & *\\
8 & 15 & 5 & 2 & * & * & *
\end{array}\right)\left(\begin{array}{c}
54\\\hline
32\\32\\32\\\hline
30\\30\\30
\end{array}\right)
\end{equation*}
}
\end{example}

\begin{example}{\rm
A Cartesian tri-magic labeling for the (2,3,3)-board whose corresponding labeling matrix is 
\begin{equation*}
\left(\begin{array}{cc|ccc|ccc}
* & * & 14 & 11 & 12 & 21 & 17 & 18\\
* & * & 13 & 15 & 10 & 16 & 19 & 20\\\hline
14 & 13 & * & * & * & 4 & 2 & 7\\
11 & 15 & * & * & * & 3 & 5 & 6\\
12 & 10 & * & * & * & 8 & 9 & 1\\\hline
21 & 16 & 4 & 3 & 8 & * & * & *\\
17 & 19 & 2 & 5 & 9 & * & * & *\\
18 & 20 & 7 & 6 & 1 & * & * & *
\end{array}\right)\left(\begin{array}{c}
93\\
93\\\hline
40\\40\\40\\\hline
52\\52\\52
\end{array}\right)
\end{equation*}
}
\end{example}

It is well known that magic rectangles and magic squares have wide applications in experimental designs~\cite{Phil1, Phil2, Phil3}. Thus, results on Cartesian magicness can be potential tools for use in situations yet unexplored. We end this paper with the following open problems and conjectures.

\begin{problem} Characterize Cartesian tri-magic $(1,r,r)$-boards  for $r\ge 4$.
\end{problem}

\begin{problem} Characterize Cartesian tri-magic $(p,p,r)$-boards for $r > p$ where $p$ is odd, $p\ge 3$ and $r$ is even.
\end{problem}

\begin{problem} Characterize Cartesian tri-magic $(p,r,r)$-boards for $r > p$ where $p$ is even, $p\ge 2$ and $r$ is odd.
\end{problem}

\begin{problem} Characterize Cartesian tri-magic $(p,q,r)$-boards for $r\ge q\ge p\ge 2$ where exactly two of the parameters are even.
\end{problem}

\begin{conjecture} Almost all $(p,q,r)$-boards are Cartesian bi-magic.
\end{conjecture}

\begin{conjecture} Almost all $(p,q,r)$-boards are not Cartesian magic.
\end{conjecture}

Finally, we may say a $(p,q,r)$-board is {\it Pythagorean magic} if $\{c_1,c_2,c_3\}$ is a set of Pythagorean triple. Thus, both $(1,1,1)-$ and $(1,1,2)-$boards are Pythagorean magic. The study of Pythagorean magic is another interesting and difficult research problem.

\bigskip\noindent
{\bf Acknowledgment} The first author is partially supported by Universiti Teknologi MARA (Johor Branch) Bestari Grant 1/2018 600-UiTMCJ (PJIA. 5/2).




\end{document}